\documentclass[10pt,reqno]{article}

\usepackage[b5paper,top=1.2in,left=0.9in]{geometry}
\usepackage{verbatim}
\usepackage{amssymb}
\usepackage{amsmath}
\usepackage{graphicx}
\usepackage{appendix}
\usepackage{color}
\usepackage{amsthm}
\usepackage{tikz}
\usetikzlibrary{arrows,positioning,decorations.pathmorphing,  decorations.markings}

\renewcommand{\d}{\mathrm{d}}
\newcommand{\D}{\mathrm{D}}
\newcommand{\e}{\mathrm{e}}

\newtheorem{Thm}{Theorem}[section]
\newtheorem{Lem}[Thm]{Lemma}
\newtheorem{Prop}[Thm]{Proposition}
\newtheorem{Cor}[Thm]{Corollary}
\newtheorem{Rem}[Thm]{Remark}
\newtheorem{Def}[Thm]{Definition}
\newtheorem{Con}[Thm]{Conjecture}
\newtheorem{Ex}[Thm]{Example}

\newtheorem*{MainThm}{Main Theorem}

\newtheoremstyle{named}{}{}{\itshape}{}{\bfseries}{.}{.5em}{#1 #3}
\theoremstyle{named}

\def\R{\mathbb{R}}
\def\Q{\mathbb{Q}}

\def\C{\mathbb{C}}
\def\Z{\mathbb{Z}}

\def\fG{\mathfrak{G}}

\def\fb{\mathfrak{b}}
\def\g{\mathfrak{g}}

\def\sl{\mathfrak{sl}}

\def\cA{\mathcal{A}}
\def\cB{\mathcal{B}}

\def\cF{\mathcal{F}}

\def\cH{\mathcal{H}}

\def\cN{\mathcal{N}}

\def\cP{\mathcal{P}}

\def\cT{\mathcal{T}}
\def\cU{\mathcal{U}}

\def\a{\alpha}

\def\G{\Gamma}
\def\D{\Delta}

\def\d{\delta}
\def\e{\epsilon}
\def\ze{\zeta}

\def\l{\lambda}
\def\L{\Lambda}

\def\s{\sigma}
\def\t{\tau}
\def\u{\upsilon}

\def\W{\Omega}
\def\w{\omega}

\def\bA{\textbf{A}}

\def\bb{\textbf{b}}

\def\be{\textbf{e}}

\def\bf{\textbf{f}}

\def\bH{\textbf{H}}

\def\bi{\textbf{i}}

\def\bo{\textbf{o}}

\def\bT{\textbf{T}}

\def\bU{\textbf{U}}

\def\=>{\Longrightarrow}
\def\inj{\hookrightarrow}

\def\to{\longrightarrow}

\def\ox{\otimes}
\def\o+{\oplus}
\def\bo+{\bigoplus}

\def\<{\langle}
\def\>{\rangle}
\def\({\left(}
\def\){\right)}
\def\oo{\infty}

\def\^{\wedge}
\def\+{\dagger}

\def\inv{^{-1}}
\def\half{\frac{1}{2}}

\def\dd[#1,#2]{\frac{d#1}{d#2}}
\def\del[#1,#2]{\frac{\partial #1}{\partial #2}}
\def\over[#1]{\overline{#1}}
\def\vec[#1]{\overrightarrow{#1}}

\def\tab{\;\;\;\;\;\;}

\newcommand{\til}[1]{\widetilde{#1}}
\newcommand{\what}[1]{\widehat{#1}}

\newcommand{\veca}[2][cccccccccccccccccccccccccccccccccccccccccc]{\left(\begin{array}{#1}#2 \\ \end{array} \right)}

\newcommand{\Eq}[1]{\begin{align}#1\end{align}}
\newcommand{\Eqn}[1]{\begin{align*}#1\end{align*}}

\begin{document}
\title{On tensor products of positive representations of split real quantum Borel subalgebra $\cU_{q\til{q}}(\fb_\R)$}

\author{  Ivan C.H. Ip\footnote{
         	   Center for the Promotion of Interdisciplinary Education and Research/\newline
     	  Department of Mathematics, Graduate School of Science, Kyoto University, Japan
		\newline
		Email: ivan.ip@math.kyoto-u.ac.jp
          }
}

\date{\today}

\numberwithin{equation}{section}

\maketitle

\begin{abstract}
We study the positive representations $\cP_\l$ of split real quantum groups $\cU_{q\til{q}}(\g_\R)$ restricted to the Borel subalgebra $\cU_{q\til{q}}(\fb_\R)$. We prove that the restriction is independent of the parameter $\l$. Furthermore, we prove that it can be constructed from the GNS-representation of the multiplier Hopf algebra $\cU_{q\til{q}}^{C^*}(\fb_\R)$ defined earlier, which allows us to decompose their tensor product using the theory of the ``multiplicative unitary''. In particular, the quantum mutation operator can be constructed from the multiplicity module, which will be an essential ingredient in the construction of quantum higher Teichm\"{u}ller theory from the perspective of representation theory, generalizing earlier work by Frenkel-Kim. 
\end{abstract}


\vspace{5mm}


\newpage

\section{Introduction}\label{sec:intro}

To any finite dimensional complex simple Lie algebra $\g$, Drinfeld \cite{D} and Jimbo \cite{J} defined a remarkable Hopf algebra $\cU_q(\g)$ known as the quantum group. Its representation theory has evolved into an important area with many applications in various fields of mathematics and physics. (cf. \cite{EGGHS} and reference therein.)

In classical Lie theory, one is interested in certain real subalgebras of $\g$ known as \emph{real forms}, two important cases being $\g_c$ corresponding to compact groups (e.g. $SU(n)$), and $\g_\R$ corresponding to split real groups (e.g. $SL(n,\R)$). The finite-dimensional representation theory in the compact case is well-behaved, and it is generalized nicely to the corresponding quantum group $\cU_q(\g_c)$. On the contrary, representation theory in the split-real case is much more complicated as was shown by the monumental works of Harish-Chandra. Its generalization to the quantum group level - involving self-adjoint operators on Hilbert spaces - is physically more relevant, but is still largely open due to various analytic difficulties coming from non-compactness and the use of unbounded operators.

\subsection{Positive representations}
The notion of the \emph{positive principal series representations}, or simply \emph{positive representations}, was introduced in \cite{FI} as a new research program devoted to the representation theory of split real quantum groups $\cU_{q\til{q}}(\g_\R)$. It uses the concept of modular double for quantum groups \cite{Fa1, Fa2}, and generalizes the case of $\cU_{q\til{q}}(\sl(2,\R))$ studied extensively by Teschner \textit{et al.} \cite{BT, PT1, PT2}. Explicit construction of the positive representations $\cP_\l$ of  $\cU_{q\til{q}}(\g_\R)$ associated to a simple Lie algebra $\g$ has been obtained for the simply-laced case in \cite{Ip2} and non-simply-laced case in \cite{Ip3}, where the generators of the quantum groups are realized by positive essentially self-adjoint operators. Furthermore, since the generators are represented by positive operators, we can take real powers by means of functional calculus, and we obtain the so-called \emph{transcendental relations} of the (rescaled) generators:
\Eq{\label{trans}\til{\be_i}=\be_i^{\frac{1}{b_i^2}},\tab \til{\bf_i}=\bf_i^{\frac{1}{b_i^2}}, \tab \til{K_i}=K_i^{\frac{1}{b_i^2}}.}
These important relations give the self-duality between different parts of the modular double, representing them simultaneously on the same Hilbert space. Moreover, in the non-simply-laced case, these provide new explicit analytic relations between the quantum group and its Langlands dual \cite{Ip3}.

In the case of the modular double $\cU_{q\til{q}}(\sl(2,\R))$, the positive representations are shown to be closely related to the space of conformal blocks of Liouville theory, and there are direct relations among the two and quantum Teichm\"{u}ller theory described below, related by the fusion and braiding operations \cite{NT, T}. In particular, the family of positive representations $\cP_\l$ of $\cU_{q\til{q}}(\sl(2,\R))$ is closed under taking tensor product, with the Plancherel measure $d\mu(\l)$ given by the quantum dilogarithm \cite{PT2}:
\Eq{\cP_{\l_1}\ox \cP_{\l_2}\simeq \int_{\R_+} \cP_\l d\mu(\l).}
The fusion relations of Liouville theory are then provided precisely by the decomposition of the triple tensor products \cite{PT1}. Together with the existence of a universal $R$-operator (constructed for general $\cU_{q\til{q}}(\g_\R)$ in \cite{Ip4}) giving the braiding structure, the braided tensor category structure may give rise to a new class of topological quantum field theory (TQFT) in the sense of Reshetikhin-Turaev \cite{RT1, RT2, Wi}. Therefore one of the major remaining unsolved problem in the theory of positive representations for higher rank is the structure under taking tensor products, and we expect that the decomposition is related to the corresponding fusion relations of more general non-compact CFT's such as the Toda conformal field theory \cite{FL, Wy}.

\subsection{Quantum Teichm\"uller theory}
The \emph{Teichm\"uller space} $\cT_S$ of an oriented surface $S$ is the space of all complex structures on $S$ modulo diffeomorphisms isotopic to identity. It is a very important space closely related to the moduli space of Riemann surfaces, and carries a natural action of the \emph{mapping class group} $\G_S$, i.e., the group of all orientation-preserving diffeomorphisms modulo isotopy, that preserves a canonical Poisson structure called the \emph{Weil-Petersson form} \cite{A}.

Hence, quantum Teichm\"uller theory is roughly speaking the quantization $\cT_S^q$ of the Poisson manifold $\cT_S$, such that the non-commutative algebra of function is represented on some Hilbert space of states $\cH$, and automorphisms of $\cT_S^q$ associated to $g\in \G_S$ is represented by unitary operators $\rho(g)$ on $\cH$. Thus, the main goal is to construct new
projective unitary representations of the mapping class group $\G_S$.

In \cite{FK}, Frenkel-Kim constructed the quantum Teichm\"{u}ller space from the perspective of representation theory of certain Hopf algebra, namely, the modular double of the quantum plane. More precisely, let $0<b<1, b\in \R\setminus\Q$ and 
\Eq{q=e^{\pi ib^2}, \tab \til{q}=e^{\pi ib^{-2}}.}
Then the modular double of the quantum plane $\cB_{q\til{q}}$ is a non-compact version of the quantum torus generated by two sets of commuting generators $\{E,K\}$ and $\{\til{E}, \til{K}\}$ such that
\Eq{KE=q^2 EK, \tab \til{K}\til{E}=\til{q}^2\til{E}\til{K}.}
The quantum plane has a canonical representation on $\cH\simeq L^2(\R)$ such that the generators $E,K$ are represented by positive self-adjoint operators
\Eq{E=e^{-2\pi bp},\tab K=e^{-2\pi bx}, \tab \til{E}=e^{-2\pi b\inv p},\tab \til{K}=e^{-2\pi b\inv x},}
with $p=\frac{1}{2\pi i}\del[,x]$. Using this fact, Frenkel-Kim \cite{FK} showed that $\cH$ is closed under taking the tensor product and decomposes as
\Eq{\cH\ox \cH \simeq M\ox \cH,}
where $M\simeq Hom_{B_{q\til{q}}}(\cH, \cH\ox \cH)$ is the multiplicity module with $\cB_{q\til{q}}$ acting trivially.

Upon identification of $\cH$ and $M$ with $L^2(\R)$, this can be expressed by a transformation using the quantum dilogarithm (cf. \eqref{FKgb}). Then the canonical isomorphism
\Eq{(\cH_1\ox \cH_2)\ox \cH_3\simeq \cH_1\ox (\cH_2\ox \cH_3)}
yields an operator $\bT$ on the multiplicity modules of the corresponding tensor product decomposition, called the \emph{quantum mutation operator}, which by construction satisfies the pentagon relation \cite[Proposition 4.8]{FK}:
\Eq{\bT_{23}\bT_{12}=\bT_{12}\bT_{13}\bT_{23},}
where the standard legged notation has been used. Together with another operator $\bA$ where $\bA^3=1$ coming from the identification of the dual representations from the antipode $S$ of $\cB_{q\til{q}}$, one recovers Kashaev's projective representation of the mapping class groupoid $\fG$ \cite{Ka2} which descends to a projective representation of $\G_S$, and we can apply it to quantization of Teichm\"{u}ller space for various surfaces as shown e.g. in \cite{CF, Fo, Ka2}. The operators $\bT$ and $\bA$ correspond essentially to the fusion and braiding relations on the conformal field theory side.

\subsection{Main results}

We observe that the quantum plane is just the Borel subalgebra of $\cU_{q\til{q}}(\sl(2,\R))$, and the canonical representation on $\cH$ is equivalent to the restriction of the positive representation. Therefore in order to generalize to higher rank, a natural candidate is to consider the Borel subalgebra $\cU_{q\til{q}}(\fb_\R)$ of the modular double of more general split real quantum groups $\cU_{q\til{q}}(\g_\R)$. It follows that one of the main ingredient needed to construct a version of quantum higher Teichm\"{u}ller theory \cite{FG1, FG2} is the construction of the quantum mutation operator $\bT$ described above, coming from the decomposition of tensor products of positive representations of the Borel part. 

In this paper, we prove the following theorem (Theorem \ref{removeL}, Corollary \ref{maincor})
\begin{MainThm} The positive representations $\cP_\l$ of $\cU_{q\til{q}}(\g_\R)$ restricted to the Borel part $\cU_{q\til{q}}(\fb_\R)$ is independent of the parameters $\l$. Moreover, the tensor product of $\cP_\l\simeq \cP$ decomposes as
\Eq{\cP\ox \cP\simeq M\ox \cP,}
where $M\simeq L^2(\R^N)$ is the multiplicity module with $\cU_{q\til{q}}(\fb_\R)$ acting trivially. Here $N$ is the dimension of $\fb_\R$.
\end{MainThm}

In particular, this decomposition provides strong evidence for the closure of positive representation of the full quantum group $\cU_{q\til{q}}(\g_\R)$ under taking tensor product. Using $\cP_\l\simeq \cP$, we also generalize in Conjecture \ref{SvNThm} the Stone-von Neumann's Theorem to the case of Borel subalgebra of higher rank split real quantum groups.

To prove the Main Theorem, we give a new construction of positive representations of the Borel part by means of \emph{multiplier Hopf algebras} introduced in \cite{KV1, VD}. We use the $C^*$-algebraic version of the Borel part $\cA:=\cU_{q\til{q}}^{C^*}(\fb_\R)$ constructed in \cite{Ip4, Ip5}, and apply the theory of Gelfand-Naimark-Segal (GNS) representations of $C^*$-algebras. This generalizes the harmonic analysis of the quantum plane studied in \cite{Ip1}, which was introduced for the harmonic analysis of the quantum space of functions $L^2(SL_{q\til{q}}^+(2,\R))$. A useful consequence of the GNS construction is the existence of a unitary operator $W$, known as the \emph{multiplicative unitary} \cite{BS}, which gives the desired intertwiner
\Eq{W\D(x) W^* = 1\ox x,\tab x\in \cA}
between the tensor product and its decomposition with a trivial multiplicity module. Therefore in order to apply this to the positive representations $\cP_\l\simeq \cP$ restricted to the Borel part, it suffices to show that $\cP$ is equivalent to the GNS representation. The second main result of this paper is then the following (Theorem \ref{mainthm}):
\begin{MainThm} For any simple Lie algebra $\g$, the GNS representation $\cP_{GNS}$ constructed by left multiplication on $\cA=\cU_{q\til{q}}^{C^*}(\fb_\R)$ is unitary equivalent to the positive representation $\cP_\l\simeq \cP$ of $\cU_{q\til{q}}(\g_\R)$ restricted to its Borel part.
\end{MainThm}

Finally, we provide concrete examples for type $A_1$ and $A_2$, where the unitary transformations by the remarkable quantum dilogarithm functions also provide rich combinatorial insight into the tensor category structure of the positive representations. In particular, these equip us with explicit expressions for the quantum mutation operators $\bT$ needed in order to construct candidates for the quantum higher Teichm\"{u}ller theory. Together with the (unitary) antipode $R_S$ defined in \cite{Ip1} and the identification of the dual representations \cite{FK}, this give us the desired construction of the projective representation of Kashaev's mapping class groupoid $\fG$ which acts on a Hilbert space $L^2(\R^N)$ with higher functional dimension. This is expected to be a new class of representations for $\fG$ and details will appear elsewhere.

The paper is organized as follows. In Section \ref{sec:prelim} we fix the notations and recall some background preliminaries needed in the calculation of this paper. In Section \ref{sec:borel}, we consider the restriction of the positive representations $\cP_\l$ to the Borel subalgebra, and show that the representation is independent of the parameters $\l$. In Section \ref{sec:GNS}, we consider the GNS-representation obtained from the action by left multiplication on the $C^*$-algebraic version of the Borel subalgebra, and bring the expression into a canonical form. In Section \ref{sec:equirep}, we prove the Main Theorem by showing that the representation $\cP_{GNS}$ constructed from the GNS representation is unitary equivalent to the one constructed earlier in \cite{Ip2, Ip3}. Finally in Section \ref{sec:Ex}, we construct explicitly the tensor product decompositions for type $A_1$ and $A_2$.

\section*{Acknowledgments}
I would like to thank Hyun Kyu Kim for helpful discussions and the editors for valuable comments. This work is supported by World Premier International Research Center Initiative (WPI Initiative), MEXT, Japan while the author resided at IPMU, currently by Top Global University Project, MEXT, Japan at Kyoto University, and JSPS KAKENHI Grant Numbers JP26800004, JP16K17571.

\section{Preliminaries}\label{sec:prelim}
Throughout the paper, we will fix once and for all $q=e^{\pi \bi b^2}$ with $\bi =\sqrt{-1}$, \\$0<b^2<1$ and $b^2\in\R\setminus\Q$. We also denote by $Q=b+b\inv$.
\subsection{Definition of the modular double $\cU_{q\til{q}}(\g_\R)$}\label{sec:prelim:Uqgr}
We follow the convention used in \cite{Ip4}. Let $\g$ be a simple Lie algebra over $\C$, $I=\{1,2,...,n\}$ denotes the set of nodes of the Dynkin diagram of $\g$ where $n=rank(\g)$. Let $W$ denote the Weyl group, and $w_0\in W$ its longest element with length $l(w_0)$.

Let $q_i:=q$ if $i\in I$ corresponds to long root, and $q_i := e^{\pi i b_s^2}$ if $i\in I$ corresponds to short root, where $b_s := \frac{b}{\sqrt{2}}$ if $\g$ is doubly-laced, and $b_s:=\frac{b}{\sqrt{3}}$ if $\g$ is of type $G_2$. We assume the roots are long for simply-laced $\g$, and denote by $b_l:=b$ as well. We let $1\in I$ corresponds to short root in type $B_n$ and long root in other types.

\begin{Def}\cite{D,J} 
The Drinfeld-Jimbo quantum group $\cU_q(\g_\R)$ is the algebra generated by $\{E_i,F_i,K_i^{\pm1}\}_{i\in I}$ over $\C$ subjected to the relations for $i,j\in I$:
\Eq{
K_iE_j=q_i^{a_{ij}}E_jK_i,\tab K_iF_j=q_i^{-a_{ij}}F_jK_i,\tab
{[E_i,F_j]} = \d_{ij}\frac{K_i-K_i\inv}{q_i-q_i\inv},
}
together with the Serre relations for $i\neq j$:
\Eq{
\sum_{k=0}^{1-a_{ij}}(-1)^k\frac{[1-a_{ij}]_{q_i}!}{[1-a_{ij}-k]_{q_i}![k]_{q_i}!}X_i^{k}X_jX_i^{1-a_{ij}-k}&=0,\tab X=E,F\label{SerreE}
}
where $[k]_q:=\frac{q^k-q^{-k}}{q-q\inv}$ and $(a_{ij})$ denotes the Cartan matrix. 
\end{Def}

We choose the Hopf algebra structure of $\cU_q(\g)$ to be given by (we will not need the counit and antipode in this paper):
\Eq{
\D(E_i)=&1\ox E_i+E_i\ox K_i,\\
\D(F_i)=&K_i^{-1}\ox F_i+F_i\ox 1,\\
\D(K_i)=&K_i\ox K_i.
}
We define $\cU_q(\g_\R)$ to be the real form of $\cU_q(\g)$ induced by the star structure
\Eq{E_i^*=E_i,\tab F_i^*=F_i,\tab K_i^*=K_i.}
Finally, from the results of \cite{Ip2,Ip3}, let
\Eq{
\til{q}:=e^{\pi \bi b_s^{-2}}
} and we define the modular double $\cU_{q\til{q}}(\g_\R)$ to be
\Eq{\cU_{q\til{q}}(\g_\R):=\cU_q(\g_\R)\ox \cU_{\til{q}}(\g_\R)&\tab \mbox{$\g$ is simply-laced,}\\
\cU_{q\til{q}}(\g_\R):=\cU_q(\g_\R)\ox \cU_{\til{q}}({}^L\g_\R)&\tab \mbox{otherwise,}}
where ${}^L\g_\R$ is the Langlands dual of $\g_\R$ obtained by interchanging the long roots and short roots of $\g_\R$.

\subsection{Positive representations of $\cU_{q\til{q}}(\g_\R)$}\label{sec:prelim:pos}
In \cite{FI, Ip2,Ip3}, a special class of representations for $\cU_{q\til{q}}(\g_\R)$, called the positive representations, is defined. The generators of $\cU_{q\til{q}}(\g_\R)$ are realized by positive essentially self-adjoint operators. They also satisfy the \emph{transcendental relations} \eqref{transdef} below, in particular the quantum group and its modular double counterpart are represented on the same Hilbert space, generalizing the situation of $\cU_{q\til{q}}(\sl(2,\R))$ introduced in \cite{Fa1, Fa2} and studied in \cite{PT2}. More precisely, 
\begin{Thm}\cite{FI, Ip2, Ip3} Let the rescaled generators be
\Eq{\be_i:=2\sin(\pi b_i^2)E_i,\tab \bf_i:=2\sin(\pi b_i^2)F_i.\label{smallef}}
There exists a family of representations $\cP_{\l}$ of $\cU_{q\til{q}}(\g_\R)$ parametrized by the $\R_+$-span of the cone of positive weights $\l\in P_\R^+$, or equivalently by $\l\in \R_+^{rank(\g)}$, such that 
\begin{itemize}
\item The generators $\be_i,\bf_i,K_i$ are represented by positive essentially self-adjoint operators acting on $L^2(\R^{N})$ where $N=l(w_0)$.
\item Define the transcendental generators:
\Eq{\til{\be_i}:=\be_i^{\frac{1}{b_i^2}},\tab \til{\bf_i}:=\bf_i^{\frac{1}{b_i^2}},\tab \til{K_i}:=K_i^{\frac{1}{b_i^2}}.\label{transdef}}
\begin{itemize}
\item if $\g$ is simply-laced, the generators $\til{\be_i},\til{\bf_i},\til{K_i}$ are obtained by replacing $b$ with $b\inv$ in the representations of the generators $\be_i,\bf_i,K_i$. 
\item If $\g$ is non-simply-laced, then the generators $\til{E_i},\til{F_i},\til{K_i}$ with $\til{\be_i}:=2\sin(\pi b_i^{-2})\til{E_i}$ and $\til{\bf_i}:=2\sin(\pi b_i^{-2})\til{F_i}$ generate $\cU_{\til{q}}({}^L\g_\R)$.
\end{itemize}
\item The generators $\be_i,\bf_i,K_i$ and $\til{\be_i},\til{\bf_i},\til{K_i}$ commute weakly up to a sign.
\end{itemize}
\end{Thm}

The positive representations are constructed for each reduced expression of the longest element $w_0\in W$ of the Weyl group, and representations corresponding to different reduced expressions are unitary equivalent.
\begin{Def}\label{variables} Fix a reduced expression of $w_0=s_{i_1}...s_{i_N}$. Let the coordinates of $L^2(\R^N)$ be denoted by $\{u_i^k\}$ so that $i$ is the corresponding root index, and $k$ denotes the sequence this root is appearing in $w_0$ from the right. Also denote by $\{v_j\}_{j=1}^N$ the same set of coordinates counting from the left, $v(i,k)$ the index such that $u_i^k=v_{v(i,k)}$, and $r(k)$ the root index corresponding to $v_k$.
\end{Def}
\begin{Ex} The coordinates of $L^2(\R^6)$ for $A_3$ corresponding to $w_0=s_3s_2s_1s_3s_2s_3$ is given by
$(u_3^3,u_2^2,u_1^1,u_3^2, u_2^1,u_3^1)=(v_1,v_2,v_3,v_4,v_5,v_6).$
\end{Ex} 
\begin{Def}We denote by $p_u=\frac{1}{2\pi \bi}\del[,u]$ and 
\Eq{e(u)&:=e^{\pi bu},\tab [u]:=q^\half e(u)+q^{-\half}e(-u),}
then the following is positive whenever $[p,u]=\frac{1}{2\pi \bi}$:
\Eq{\label{standardform}[u]e(-2p):=(q^\half e^{\pi bu}+q^{-\half}e^{-\pi bu})e^{-2\pi bp} = e^{\pi b(u-2p)}+e^{\pi b(-u-2p)}}
\end{Def}
Note that we changed the notation slightly for $e(u)$ in this paper from previous references \cite{Ip2}-\cite{Ip4} for later convenience.

\begin{Def}\label{usul}By abuse of notation, we denote by
\Eq{[u_s+u_l]e(-2p_s-2p_l):=e^{\pi b_s(-u_s-2p_s)+\pi b_l(-u_l-2p_l)}+e^{\pi b_s(u_s-2p_s)+\pi b_l(u_l-2p_l)},}
where $u_s$ (resp. $u_l$) is a linear combination of the variables corresponding to short roots (resp. long roots). The parameters $\l_i$ are also considered in both cases. Similarly $p_s$ (resp. $p_l$) are linear combinations of the $p$ shifting of the short roots (resp. long roots) variables.

We will occasionally write in the form $[X]e(-2p_Y)$, where by abuse of notation, if $Y = \sum a_i u_i, a_i\in \Z$, then $p_Y = \sum a_i p_{u_i}$.
\end{Def}
\begin{Thm}\label{FKaction}\cite{Ip2,Ip3} For a fixed reduced expression of $w_0$, $\cP_\l$ is given by
\Eq{
\bf_i=&\sum_{k=1}^n\left[-\sum_{j=1}^{v(i,k)-1} a_{i,r(j)}v_j-u_i^k-2\l_i\right]e(2p_{i}^{k}),\label{FF}\\
K_i=&e\left(-\sum_{k=1}^{l(w_0)} a_{i,r(k)}v_k+2\l_i\right).\label{KK}
}
By taking $w_0=w's_i$ so that the simple reflection for root $i$ appears on the right, the action of $\be_i$ is given by
\Eq{\label{EE} 
\be_i=&[u_i^1]e(-2p_i^1).}
\end{Thm}

In this paper, it is instructive to recall the explicit expression in the case of rank 1 and 2. For details of the construction and the other cases please refer to \cite{Ip2,Ip3}. 

\begin{Prop}\cite{BT,PT2}\label{canonicalsl2} The positive representation $P_\l$ of $\cU_{q\til{q}}(\sl(2,\R))$ acting on $L^2(\R$) is given by 
\Eqn{
\be=&[u-\l]e(-2p)=e^{\pi b(-u+\l-2p)}+e^{\pi b(u-\l-2p)},\\
\bf=&[-u-\l]e(2p)=e^{\pi b(u+\l+2p)}+e^{\pi b(-u-\l+2p)},\\
K=&e(-2u)=e^{-2\pi bu}.
}
Note that it is unitary equivalent to \eqref{FF}-\eqref{EE} by $u\mapsto u+\l$.
\end{Prop}
\begin{Prop}\cite{Ip2}\label{typeA2} The positive representation $\cP_\l$ of $\cU_{q\til{q}}(\sl(3,\R))$ with parameters $\l=(\l_1,\l_2)\in\R_+^2$, corresponding to the reduced expression $w_0=s_2s_1s_2$, acting on $f(u,v,w)\in L^2(\R^3)$, is given by
\Eqn{
\be_1=&[v-w]e(-2p_v)+[u]e(-2p_v+2p_w-2p_u),\\
\be_2=&[w]e(-2p_w),\\
\bf_1=&[-v+u-2\l_1]e(2p_v),\\
\bf_2=&[-2u+v-w-2\l_2]e(2p_w)+[-u-2\l_2]e(2p_u),\\
K_1=&e(u-2v+w-2\l_1),\\
K_2=&e(-2u+v-2w-2\l_2).
}
\end{Prop}

\begin{Prop}\label{B2rep}\cite{Ip3} The positive representation  $\cP_\l$ of $\cU_{q\til{q}}(\g_\R)$ with parameters $\l=(\l_1,\l_2)\in\R_+^2$, where $\g_\R$ is of type $B_2$, corresponding to the reduced expression $w_0=s_1s_2s_1s_2$, acting on $f(t,u,v,w)\in L^2(\R^4)$, is given by
\Eqn{
\be_1=&[t]e(-2p_t-2p_u+2p_w)+[u-v]e(-2p_u-2p_v+2p_w)+[v-w]e(-2p_v),\\
\be_2=&[w]e(-2p_w),\\
\bf_1=&[-t-2\l_1]e(2p_t)+[-2t+u-v-2\l_1]e(2p_v),\\
\bf_2=&[2t-u-2\l_2]e(2p_u)+[2t-2u+2v-w-2\l_2]e(2p_w),\\
K_1=&e(-2t+u-2v+w-2\l_1),\\
K_2=&e(2t-2u+2v-2w-2\l_2).
}
In this case (cf. Definition \ref{usul}), $u_s$ is linear combinations of $\{t,v\}$, while $u_l$ is linear combinations of $\{u,w\}$. Similarly for $p_s$ and $p_l$. 
\end{Prop}

We will omit the case of type $G_2$ in this paper for simplicity.

\subsection{Quantum dilogarithm $G_b(x)$ and $g_b(x)$}\label{sec:prelim:qdlog}
First introduced by Faddeev and Kashaev \cite{Fa1, FKa}, the quantum dilogarithm $G_b(x)$ and its variants $S_b(x)$ and $g_b(x)$ play a crucial role in the study of positive representations of split real quantum groups, and also appear in many other areas of mathematics and physics, most notably cluster algebras, quantum Teichm\"{u}ller theory and Liouville CFT. In this subsection, we recall the definition and some properties of the quantum dilogarithm functions \cite{BT, Ip1, PT2} that is needed in the calculations of this paper.

\begin{Def} The quantum dilogarithm function $G_b(x)$ is defined on\\ ${0\leq Re(x)\leq Q=b+b\inv}$ by
\Eq{\label{intform} G_b(x)=\over[\ze_b]\exp\left(-\int_{\W}\frac{e^{\pi tx}}{(e^{\pi bt}-1)(e^{\pi b\inv t}-1)}\frac{dt}{t}\right),}
where $\ze_b=e^{\frac{\pi \bi }{2}(\frac{b^2+b^{-2}}{6}+\frac{1}{2})}$,
and the contour $\W$ goes along $\R$ and goes above the pole at $t=0$. This can be extended meromorphically to the whole complex plane with poles at $x=-nb-mb\inv$ and zeros at $x=Q+nb+mb\inv$, for $n,m\in\Z_{\geq0}$.
\end{Def}

\begin{Def} \label{Sb}The function $S_b(x)$ is defined by
\Eq{S_b(x) := e^{\frac{\pi \bi}{2} x(Q-s)}G_b(x).}
\end{Def}

\begin{Prop}
The quantum dilogarithms satisfy the self-duality relations:
\Eq{S_b(x) = S_{b\inv}(x),\tab G_b(x)=G_{b\inv}(x),\label{selfdual}}
as well as the functional equations
\Eq{\label{funceq}S_b(x+b^{\pm 1}) &= i(e^{-\pi\bi b^{\pm1} x}-e^{\pi\bi b^{\pm1} x})S_b(x)\\
 \tab G_b(x+b^{\pm 1})&=(1-e^{2\pi \bi b^{\pm 1}x})G_b(x).}
\end{Prop}

We will also need another important variant of the quantum dilogarithm.
\begin{Def} The function $g_b(x)$ is defined by
\Eq{g_b(x)=\frac{\over[\ze_b]}{G_b(\frac{Q}{2}+\frac{\log x}{2\pi \bi b})},}
where $\log$ takes the principal branch of $x$.
\end{Def}

We will need the following properties of $g_b(x)$:
\begin{Lem}\label{FT} \cite[(3.31), (3.32)]{BT} We have the following Fourier transformation formula:
\Eq{\int_{\R+\bi 0} \frac{e^{-\pi \bi  t^2}}{G_b(Q+\bi t)}X^{\bi b\inv t}dt&=g_b(X),\\ \label{FT1}
\int_{\R+\bi 0} \frac{e^{-\pi Qt}}{G_b(Q+\bi t)}X^{\bi b\inv t}dt&=g_b^*(X) ,} where $X$ is a positive operator and the contour goes above the pole
at $t=0$.
\end{Lem}

\begin{Lem}\label{unitary} $g_b(X)$ is a unitary operator for any positive operator $X$.
\end{Lem}

\begin{Lem}\label{qsum}If $UV=q^2VU$ where $U,V$ are positive self adjoint operators, then
\begin{eqnarray}
\label{qsum0}g_b(U)g_b(V)&=&g_b(U+V),\\
\label{qsum1}g_b(U)^*Vg_b(U) &=& q\inv UV+V,\\
\label{qsum2}g_b(V)Ug_b(V)^*&=&U+q\inv UV.
\end{eqnarray}
Note that \eqref{qsum0} and \eqref{qsum1} together imply the pentagon relation
\Eq{\label{qpenta}g_b(V)g_b(U)=g_b(U)g_b(q\inv UV)g_b(V).}
\end{Lem}
The above relations can be generalized to the non-simply-laced case \cite[Prop 3.3]{Ip4} which give relations between $g_{b_s}$ and $g_{b_l}$.
\subsection{Lusztig's isomorphism}\label{sec:prelim:Ti}
We recall several definitions and results concerning the Lusztig's isomorphism in the positive representation setting needed in this paper. See \cite{Ip4} for more details.

Fix a positive representation $\cP_\l$ of $\cU_{q\til{q}}(\g_\R)$., and denote by 
\Eq{[u,v]_q:=quv-q\inv vu.}

\begin{Prop}\label{eij} In the simply-laced case, the operators
\Eq{\be_{ij}:=\frac{[\be_j,\be_i]_{q^\half}}{q-q\inv}=\frac{q^{\half}\be_j\be_i-q^{-\half}\be_i\be_j}{q-q\inv}\label{eij1}}
are positive essentially self-adjoint on $\cP_\l$.
\end{Prop}

\begin{Prop}\label{eij2} In the non-simply-laced case, the operators
\Eq{\be_{ij}:=&(-1)^{a_{ij}}\left[\left[\be_i,...[\be_i,\be_j]_{q_i^{\frac{a_{ij}}{2}}}\right]_{q_i^{\frac{a_{ij}+2}{2}}}...\right]_{q_i^{\frac{-a_{ij}-2}{2}}}\prod_{k=1}^{-a_{ij}}(q_i^k-q_i^{-k})\inv\label{eij3}}
are positive essentially self-adjoint on $\cP_\l$.
\end{Prop}
Similarly we define $\bf_{ij}$ by \eqref{eij1} and \eqref{eij3} with $\be$ replaced by $\bf$.
\begin{Thm}\label{LusThm} For any root $i\in I$, there exists a unitary operator $T_i$ such that
\Eq{
T_i(\be_i)=&q_i \bf_iK_i\inv=q_i\inv K_i\inv \bf_i,\\
T_i(\bf_i)=&q_i\inv K_i \be_i=q_i\be_i K_i,\\
T_i(\be_j)=&\be_{ij}, \tab T_i(\bf_j)=\bf_{ij},\tab\mbox{ for $i,j$ adjacent,}\\
T_i(\be_k)=&\be_k,\tab T_i(\bf_k)=\bf_k,\tab\mbox{ for $a_{ik}=0$,}\\
T_i(K_j)=&K_jK_i^{-a_{ij}}.
}
\end{Thm}
\begin{Prop}\label{LuCox}\cite{Lu1, Lu2} The operators $T_i$ satisfy the Coxeter relations:
\Eq{\underbrace{T_iT_jT_i...}_{-a_{ij}'+2} = \underbrace{T_j T_i T_j...}_{-a_{ij}'+2}.\label{coxeter},}
where $-a'_{ij}=\max\{-a_{ij},-a_{ji}\}$.
Furthermore, for $\a_i,\a_j$ simple roots, and an element $w=s_{i_1}...s_{i_k}\in W$ such that $w(\a_i)=\a_j$, we have for $X=\be,\bf,K$:
\Eq{T_{i_1}...T_{i_k}(X_i)=X_j.}
\end{Prop}

In this paper we will use instead the inverse version of the following definition.
\begin{Def} Let $w_0=s_{i_1}s_{i_2}...s_{i_N}$ be the longest element of $W$. We define
\Eq{\be_{\a_k}:=T_{i_1}\inv T_{i_2}\inv... T_{i_{k-1}}\inv \be_{i_k}.\label{eak}}
These are all positive essentially self-adjoint.
\end{Def}

Finally we record some commutation relations among the non-simple generators.
\begin{Prop}\label{nonsimple_rel}For the simply-laced case, if $a_{ij}=-1$, we have
\Eq{\be_i\be_{ij}&= q\inv\be_{ij}\be_i, \tab \be_j\be_{ij}=q\be_{ij}\be_j, \\
\be_i\be_j^n &= q^n\be_j^n\be_i + q^{\half}(q^{-n}-q^{n})\be_j^{n-1}\be_{ij}.\label{eijcom}}
\end{Prop}
\begin{Prop}\label{nonsimple_rel2}For doubly-laced case, let $a_{ij}=-2$. Denote $\be_{ij}:=T_i\be_j$ and $\be_X:=T_iT_j\be_i$. We have
\Eq{
\be_i\be_{ij}&=q\inv \be_{ij}\be_i,\tab \be_{ij}\be_X=q\inv \be_X\be_{ij},\tab \be_X\be_j =q\inv\be_j\be_X\\
\be_i^n\be_j&=q^{-n}\be_j \be_i^n +q^{\frac{1}{2}}(q^n-q^{-n}) \be_i^{n-1}\be_X,\\
\be_i^n\be_X&=\be_X\be_i^n+q^{\frac{1}{2}}(q^{-n}-1)\be_{ij}\be_i^{n-1},\\
\be_{ij}^n\be_j&=q^n\be_j\be_{ij}+(q^{-2n}-1)\be_X^2.
}
\end{Prop}
Note that by means of \cite[Prop. 6.8]{Ip4}, we can take $n$ to be a complex power.
\subsection{GNS representation, multiplicative unitary and multiplier Hopf algebra}\label{sec:GNS}
Finally, we recall the definition of Gelfand-Naimark-Segal (GNS) representation of a $C^*$-algebra, and its corresponding unitary operator called the \emph{multiplicative unitary} that is needed in this paper. These are fundamental to the theory of locally compact quantum groups in the setting of $C^*$-algebras and von Neumann algebras, and to the generalization of Pontryagin duality. In particular, a multiplicative unitary encodes all the structure maps of a quantum group and its dual. See \cite{Ip1}, \cite{KV1} or \cite{Tim} for more discussions.

\begin{Def} A Gelfand-Naimark-Segal (GNS) representation of a $C^*$-algebra $\cA$ with a weight $\phi$ is a triple
$(\cH, \pi, \L),$ where $\cH$ is a Hilbert space, $\L:\cA\to\cH$ is a linear map, and $\pi:\cA\to\cB(\cH)$ is a representation of $\cA$ on $\cH$ as bounded linear operators, such that $\L(\cN)$ is dense in $\cH$, where $\cN=\{a\in\cA: \phi(a^*a)<\oo\}$ and
\begin{eqnarray}\label{right}
\pi(a)\L(b)&=&\L(ab) \tab\forall a\in\cA,b\in \cN,\\
\<\L(a),\L(b)\>&=&\phi(b^*a)\tab\forall a,b\in\cN.
\end{eqnarray}
\end{Def}

\begin{Def}\cite{BS} Let $\cH$ be a Hilbert space. A unitary operator $W\in \cB(\cH\ox \cH)$ is called a \emph{multiplicative unitary} if it satisfies the pentagon equation
\Eq{W_{23}W_{12}=W_{12}W_{13}W_{23},}
where the standard leg notation is used.
\end{Def}

Given a GNS representation $(\cH,\pi, \L)$ of a locally compact quantum group $\cA$, we can define a unitary operator on $\cH\ox \cH$ by \Eq{\label{WW}W^*(\L(a)\ox \L(b)) := (\L\ox \L)(\D(b)(a\ox 1)).}

It is known that $W$ is a multiplicative unitary \cite[Thm 3.16, Thm 3.18]{KV1}, and the coproduct on $\cA$ defining it can be recovered from $W$:
\begin{Prop}\label{WD} Let $x\in\cA \inj \cB(\cH)$ as operator. Then 
\Eq{W^*(1\ox x)W = \D(x)} as operators on $\cH\ox \cH$.
\end{Prop}
\begin{proof}See \cite[Prop 2.18]{Ip1} 
\end{proof}
In particular, Proposition \ref{WD} provides a unitary equivalence 
\Eq{W: \cH\ox \cH \simeq M\ox \cH}
of the tensor product of the GNS representations, where $M$ is the multiplicity module equipped with the trivial action. This is the key step to proving the tensor product decomposition of the positive representations of $\cU_{q\til{q}}(\fb_\R)$.

Finally, we recall the definition of a multiplier Hopf algebra.
\begin{Def}
The multiplier algebra $M(\cA)$ of a $C^*$-algebra $\cA\subset \cB(\cH)$ is the $C^*$-algebra of operators
\Eq{M(\cA)=\{b\in \cB(\cH): b\cA\subset \cA, \cA b\subset \cA\}.}
In particular, $\cA$ is an ideal of $M(\cA)$.
\end{Def}
\begin{Def}\label{mHa}
A multiplier Hopf *-algebra is a $C^*$-algebra $\cA$ together with the antipode $S$, the counit $\e$, and the coproduct map
\Eq{\D:\cA\to M(\cA\ox \cA),}
all of which can be extended to a map from $M(\cA)$, such that the usual properties of a Hopf algebra holds on the level of $M(\cA)$.
\end{Def}
\section{Positive representations restricted to Borel part}\label{sec:borel}
In the case of $\g=\sl_2$, the Borel subalgebra $\cU_{q\til{q}}(\fb_\R)$ is generated by the operator $K$ and $\be$ satisfying the quantum plane relation $K\be=q^2\be K$, in the sense of integrable representations \cite{Sch}, i.e.
\Eq{K^{\bi b\inv s}\be^{\bi b\inv t} = e^{-2\pi ist}\be^{\bi b\inv t}K^{\bi b\inv s},\tab \forall s,t\in \R}
as unitary operators. Thus the positive representation $\cP_\l\simeq L^2(\R)$ restricted to the Borel part is nothing but an irreducible representation of the quantum plane, such that the operators are represented by positive self-adjoint operators. By Stone-von Neumann's Theorem, we know that every such representation is unitary equivalent to the canonical representation $\cH\simeq L^2(\R)$ given by 
\Eq{K=e^{-2\pi b x},\tab \be=e^{-2\pi bp}.} 
In particular, the representation does not depend on any parameters. In fact, from Proposition \ref{canonicalsl2}, it is easily seen that $\Phi: \cP_\l\to \cH$ given by multiplication by
\Eq{\Phi = e^{-\frac{\pi \bi (x-\l)^2}{2}}g_b(e^{-2\pi b(x-\l)})}
is the required unitary equivalence. This result also holds in the higher rank case:

\begin{Thm} \label{removeL}In the higher rank case, we have \Eq{\cP_{\l}\simeq \cP_{\l'}} as representations of $\cU_{q\til{q}}(\fb_\R)$ for any $\l,\l'\in \R^{rank(\g)}$.
\end{Thm}
\begin{proof}
From \cite[Section 11]{Ip2} for the simply-laced case, or \cite[Section 9]{Ip3} for non-simply-laced case, we know that $\cP_\l\simeq \cP_{w(\l)},$ where $w\in W$ is any Weyl group element acting on $\l$, namely for simple reflections,
\Eq{s_i(\l_j) := \l_j - a_{ij}\l_i,} 
where $a_{ij}$ is the Cartan matrix. In particular, this means that for each root $i\in I$, there exists a unitary transformation $B_i$ such that it fixes the action of $\be_i$, but modify the action of $K_j$ by 
$$K_j \mapsto K_je^{a_{ij}\pi b \l_i}.$$
However, when we restrict to the Borel subalgebra, we are released from the restriction of the action of the lower Borel generated by $\bf_i$, in which the proof above allows us to define the unitary transformation $B_i$ such that it fixes the action of $\be_i$ but modify the action of $K_i$ by \emph{an arbitrary constant} $c_i$:
$$K_j \mapsto K_je^{a_{ij}\pi b c_i}.$$
Therefore, it suffices to solve the equation $\sum a_{ij} c_i = \l_i$ in order for all $\l_i$ in the representation to cancel. This clearly can be done because the Cartan matrix is invertible. In particular, all representations $\cP_\l$ are unitary equivalent to $\cP_{\vec[0]}$.

\end{proof}

However, we observe that $\cP_{\vec[0]}$ is not irreducible in general. For example, in the case of type $A_2$, the positive operator $K_1^{\frac23}K_2^{-\frac23}\be_{12}\inv\be_{21}$ commutes with all the generators, hence $\cP_{\vec[0]}$ decomposes as a direct integral $\int_{\R_{>0}}^\o+ P_s ds$ with respect to the spectrum $s\in\R_{>0}$ of this central operator, where each $P_s\simeq L^2(\R^2)$ is irreducible. 

\begin{Con} \label{SvNThm}Any irreducible representations of $\cU_{q\til{q}}(\fb_\R)$ satisfying the conditions of a positive representation acts on $P\simeq L^2(\R^{rank(\g)})$, and is unitary equivalent to a direct integral summand of $\cP_{\vec[0]}$. More specifically, using the notion of ``integrable representation" of quantum plane \cite{Sch}, we require as unitary operators
$$K_i^{\bi b\inv s}\be_j^{\bi b\inv t} = e^{-\pi a_{ij}st}\be_j^{\bi b\inv t}K_i^{\bi b\inv s}\tab \forall s,t\in\R,$$
and in the simply-laced case,
$$\be_i^{\bi b\inv s}\be_{ij}^{\bi b\inv t} = e^{\pi st}\be_{ij}^{\bi b\inv t}\be_i^{\bi b\inv s}\tab \forall s,t\in\R,$$
$$\be_j^{\bi b\inv s}\be_{ij}^{\bi b\inv t} = e^{-\pi st}\be_{ij}^{\bi b\inv t}\be_j^{\bi b\inv s}\tab \forall s,t\in\R,$$
which corresponds to the Serre's relations. We have similar equations for the non-simply-laced case (cf. Proposition \ref{nonsimple_rel2}). This generalizes the Stone von-Neumann Theorem to the case of Borel subalgebra of a split real quantum group. 
\end{Con}

\section{GNS representation of $\cU_{q\til{q}}^{C^*}(\fb_\R)$}\label{sec:GNS}
In this section, we recall the definition of $\cU_{q\til{q}}^{C^*}(\fb_\R)$ and construct the GNS-representation
by left multiplication. We show that it is equivalent to a positive representation.
\subsection{Definition of $\cU_{q\til{q}}^{C^*}(\fb_\R)$}\label{sec:defcstar}
Let us denote by $\cP:=\cP_{\vec[0]}$. In \cite{Ip4}, we used the language of multiplier Hopf algebra, and defined a $C^*$-algebraic version of the Borel subalgebra using a continuous basis. Explicitly, consider the action of $\cU_{q\til{q}}(\fb_\R)$ on $\cP_\l\simeq \cP\simeq L^2(\R^N)$, where $N=l(w_0)$. 
\begin{Def} \label{Ub}
Let $n=rank(\g)$. We define the $C^*$-algebraic version of the Borel subalgebra $\bU\bb:=\cU_{q\til{q}}^{C^*}(\fb_\R)$ as the operator norm closure of the linear span of all bounded operators on $L^2(\R^N)$ of the form
\Eq{\label{vecF}\vec[F]:= \left(\prod_{k=1}^N  \int_C \frac{F_k(t_k)}{G_{b_{i_k}}(Q_{i_k}+\bi t_k)} \be_{\a_k}^{\bi b_{i_k}\inv t_k}dt_k\right)\left(\prod_{i=1}^n \int_\R \what{F}_i(s_i)K_i^{\bi b\inv s_i}ds_i\right),}
where $\be_{\a_k}$ is given by \eqref{eak} with the order of multiplication as $\be_{\a_1}...\be_{\a_N}$, $F_k(t_k)$ and $\what{F}_i(s_i)$ are both entire analytic functions that rapidly decay along the real direction, i.e. for fixed $y_0$, $F_k(x+\bi y_0)$ decays faster than any exponential function in $x$. Finally the contour $C$ goes along the real axis and above the pole of $G_b$ at $t_k=0$.
\end{Def}
It was shown in \cite{Ip4} that $\bU\bb$ is a multiplier Hopf algebra (cf. Definition \ref{mHa}).

\begin{Rem} We modified the definition slightly from \cite{Ip4}. We have used the opposite multiplication order to be consistent with the definition of $\be_{\a_k}$ using $T_{i_k}\inv$ (cf. \eqref{eak}) instead of $T_{i_k}$. The two definitions are identical. Moreover, to simplify calculations, we have used the functions $\what{F}_i(s_i)$ parametrized by the powers of $K_i$ instead of a general compact function $F_0(\bH)$ considered in \cite{Ip4}, where $K_i = q_i^{H_i}$. These are simply related by Fourier transform $\cF$:
\Eqn{\int (\cF \what{F})(s) K^{\bi b\inv s} ds= \iint e^{-2\pi\bi s t} \what{F}(t) K^{\bi b\inv s} dt ds= \what{F}(\frac{1}{2}\bi bH).}
\end{Rem}

\subsection{GNS-representation}\label{sec:GNSrep}
It turns out the definition above does not necessarily define a positive action. In order to obtain an action of $\cU_{q\til{q}}(\fb_\R)$ that is positive, we modify $\vec[F]\in \cA$ as follows.
\begin{Def}\label{Ubnew} We define the GNS representation $(\cH, \pi, \L)$ of $\cA:=\cU_{q\til{q}}^{C^*}(\fb_\R)$ on $\cH\simeq L^2(\R^{N+n})$ as follows. We let
\Eq{\vec[F]:=\prod_{k=1}^N  \left( \int_C \frac{F_k(t_k)}{S_{b_{i_k}}(Q_{i_k}+\bi t_k)} \be_{\a_k}^{\bi b_{i_k}\inv t_k} dt_k K_{\a_k}^{\frac{Q}{2b}}\right)\left(\prod_{i=1}^n \int_\R \what{F}_i(s_i)K_i^{\bi b\inv s_i}e^{\pi Q s_i}ds_i\right),}
and define the representation $\L:\cA\to \cH$ by
\Eq{
\vec[F]&\mapsto\prod_{k=1}^N F_k(t_k)\prod_{i=1}^n \what{F}_i(s_i)}
and extend by linearity. Finally, the $L^2$-norm of $\cH$ (cf. \cite[Thm 4.9]{Ip1}) will be defined for each variable as
\Eq{\|F(..., t_k,...)\|^2:=& \int_{\R^N} \left|F\left(..., t_k+\frac{\bi Q_{i_k}}{2},...\right)\right|^2dt_1...dt_N,\label{haar}\\
\|\what{F}(..., s_i, ...)\|^2:=& \int_{\R^n} |\what{F}(..., s_i, ...)|^2 ds_1...ds_n.}
\end{Def}

Note that we have used $S_b(x)$ (cf. Definition \ref{Sb}) instead in the denominator. Moreover we slipped $K_{\a_k}$ in between the products of $\be_{\a_k}$, this is well defined since $K_i^{\frac{Q}{2b}}$ is in the multiplier $M(\cA)$ of $\cA$. Finally the $L^2$-norm \eqref{haar} can be thought of as shifting the arguments by $\frac{\bi Q}{2}$ and use the usual $L^2$-norm. More precisely,
\Eq{\cH&\to L^2(\R^{N+n}, dt_1...dt_Nds_1...ds_n)\nonumber\\
F(..., t_k, ...)&\mapsto F'(..., t_k, ...) := F\left(..., t_k+\frac{\bi Q_{i_k}}{2}, ...\right)}
is an isomorphism of Hilbert space. Therefore, to obtain a representation on the usual $L^2$-norm, we work on $F$, and at the end shift all the variables $t_k$ by this map.

\begin{Prop} \label{ignoreS}The GNS representation is unitary equivalent to a representation that is independent on the $s_i$ variables.
\end{Prop}
\begin{proof}
The action of multiplication by $\be_i$ from the left will not depend on $s_i$. Moreover, multiplication of $K_i$ from the left, after commuting with all the $\be_{\a_k}$, give us an extra factor $e^{-2\pi bp_{s_i}}$, which is independent from the $t_i$ variables. Therefore these $p_{s_i}$ can be treated as constants $\l_i$. Once we have shown that the $t_k$ part is unitary equivalent to the representation $\cP$ in Section \ref{sec:equirep}, by Proposition \ref{removeL} there exists a unitary transformation that can remove the dependence on $s_i$ completely. 
\end{proof}
In particular the GNS representation is unitary equivalent to
\Eq{\cH = L^2(\R^{N+n}) \simeq \cP \ox M_K\label{MK}}
for a trivial multiplicity module $M_K\simeq L^2(\R^n)$. Therefore without loss of generality, we will ignore the $K_i^{\bi b\inv s_i}$ from now on, and the GNS representation calculated below will assume to act trivially on the $s_i$ variables.
Note that the $\be_{\a_k}$ part already gives the correct functional dimension.

\subsection{Rank 2 case}
We illustrate the representation in the rank 2 case, which is essential to generalize to the higher rank.
Let us first consider the simply-laced case, and fix the longest element to be $w_0= s_2s_1s_2$. Also using Definition \ref{variables} and \ref{Ub}, we name the variables $$(u,v,w):=(v_1,v_2,v_3)=(t_3,t_2,t_1)= (u_2^2,u_1^1,u_2^1).$$

Then a typical element of $\vec[F]\in \cA$ (ignoring the ending $K_i$ terms) is given by 
\Eq{\vec[F]&=\int_C F(u,v,w)\frac{\be_2^{\bi b\inv w}K_2^{\frac{Q}{2b}}}{S_b(Q+\bi w)}\frac{\be_{12}^{\bi b\inv v}(K_1K_2)^{\frac{Q}{2b}}}{S_b(Q+\bi v)}\frac{\be_1^{\bi b\inv u}K_1^{\frac{Q}{2b}}}{S_b(Q+\bi u)}dudvdw\nonumber\\
&=\int_C F(u,v,w)e^{-\frac{\pi Qv}{2}}\frac{\be_2^{\bi b\inv w}}{S_b(Q+\bi w)}\frac{\be_{12}^{\bi b\inv v}}{S_b(Q+\bi v)}\frac{\be_1^{\bi b\inv u}}{S_b(Q+\bi u)}dudvdw}
where in the last line we moved the $K_i$ to the right and absorb into the (ignored) $K$ terms. Note that we have reversed the order of the variables since we used the opposite multiplication order (reading from left to right).

Let us illustrate how to find the actions of the generators $\be_i$ in general. The action of $\be_2$ is easy. It only shifts the variables $w$. We have (ignoring other variables)

\Eqn{\be_2\cdot \vec[F] &= \int F(u,v,w)...\frac{\be_2^{\bi b\inv w+1}}{S_b(Q+\bi w)} ...dw\\
&=\int  F(u,v,w+\bi b)...\frac{\be_2^{\bi b\inv w}}{S_b(Q+\bi w -b)}... dw\\
&=-i\int (e^{\pi bw}-e^{-\pi bw})F(u,v,w+\bi b)...\frac{\be_2^{\bi b\inv w}}{S_b(Q+\bi w)}... dw,}
where we have used property \eqref{funceq} of $S_b$. Hence the action of $\be_2$ on $\cH$ (i.e. on the coefficient function) is given by
\Eq{\be_2 = -\bi (e^{\pi bw}-e^{-\pi bw})e^{-2\pi bp_w}.}
Finally, the shift $w\mapsto \frac{\bi Q}{2}$ gives us the positive expression on $L^2(\R^3)$:
\Eq{\be_2&=-\bi (e^{\pi b(w+\frac{\bi Q}{2})}-e^{-\pi b(w+\frac{\bi Q}{2})})e^{-2\pi bp_w}\nonumber\\
&=-\bi (iq^\half e^{\pi bw} + iq^{-\half} e^{-\pi bw})e^{-2\pi bp_w}\nonumber\\
&= e(w-2p_w)+e(-w-2p_w).}

For the action of $\be_1$, we used Proposition \ref{nonsimple_rel} and proceed as above.
\Eq{\be_1\be_2^{\bi b\inv w} &= q^{\bi b\inv w} \be_2^{\bi b\inv w}\be_1 + q^{\half}(q^{-\bi b\inv w}-q^{\bi b\inv w})\be_2^{\bi b\inv w-1}\be_{12}\nonumber\\
&=e^{-\pi bw} \be_2^{\bi b\inv w}\be_1 + q^{\half}(e^{\pi bw}-e^{-\pi bw})\be_2^{\bi b\inv w-1}\be_{12}.\label{e1e2}}

In the first term of \eqref{e1e2} $\be_1$ commutes again with $\be_{12}^{\bi b\inv v}$ and pick up the factor $q^{-\bi b\inv v}=e^{\pi bv}$ before absorbing into $\be_1^{\bi b\inv u}$, producing the $S_b$ factor same as above. Hence after shifting by $\frac{\bi Q}{2}$ it gives overall
$$e(-w+v)(e(u-2p_u)+e(-u-2p_u)).$$
The second term of \eqref{e1e2} gives an opposite shifting $e(2p_w)$ in the variable $w$, which cancels the factor in front:
\Eqn{q^{\half}(e^{\pi b(w-ib)}-e^{-\pi b(w-ib)})\frac{S_b(Q+\bi w)}{S_b(Q+\bi w+b)}=iq^{\frac{1}{2}}.}

The shifting $e(-2p_v)$ in $v$ by $\be_{12}$ then provides the same factor as before, as well as the extra number $e^{-\frac{\pi Q(\bi b)}{2}}=(\bi q^\half)\inv$ from the auxiliary term $e^{-\frac{\pi Qv}{2}}$ which cancels the above number. After shifting by $\frac{\bi Q}{2}$ we then obtain 
$$e(v-2p_v+2p_w)+e(-v-2p_v+2p_w).$$

Overall, the action of $\be_1$ is then given by
\Eq{\be_1=&e(-w+v+u-2p_u)+e(-w+v-u-2p_u)+e(v-2p_v+2p_w)+e(-v-2p_v+2p_w).\label{e1}}

Finally, the action of $K_i$ is easy to read out. Multiplication from the left, each $K_i$ commute with the root vectors $\be_j$ and pick up a factor of the form $(q^k)^{\bi b\inv u} = e(-ku)$, before absorbing into the $K$ terms. We find that the action is given by
\Eqn{K_1 =e(-2u-v+w-2p_{s_1}+\bi Q),\tab K_2=e(u-v-2w-2p_{s_2}+\bi Q).}

After shifting $(u,v,w)$ by $\frac{\bi Q}{2}$, we have a positive action of $K_i$. Then by Proposition \ref{ignoreS}, we can ignore the $p_{s_i}$ factors.

In summary, we obtained
\Eqn{
\be_1=&e(-w+v+u-2p_u)+e(-w+v-u-2p_u)+e(v-2p_v+2p_w)+e(-v-2p_v+2p_w),\\
\be_2 =&e(w-2p_w)+e(-w-2p_w),\\
K_1 =&e(-2u-v+w),\\
K_2=&e(u-v-2w).
}
and we will show that this is unitary equivalent to the positive representation $\cP$.

We can do the same procedure similarly for non-simply-laced case. For Type $B_2$, using $w_0=s_1s_2s_1s_2$, we define
\Eq{\vec[F]=\int_C F(u,v,w)e^{-\pi Q_sv-\frac{\pi Qu}{2}}\frac{\be_2^{\bi b\inv w}}{S_b(Q+\bi w)}\frac{\be_{12}^{\bi b_s\inv v}}{S_{b_s}(Q_s+\bi v)}\frac{\be_X^{\bi b\inv u}}{S_b(Q+\bi u)}\frac{\be_1^{\bi b_s\inv t}}{S_{b_s}(Q_s+\bi t)}dtdudvdw,}
where $\be_X := T_1T_2(\be_1)$. Then using successively Proposition \ref{nonsimple_rel2}, we arrive at
\Eqn{
\be_1 &= e(-w+u+t-2p_t)+e(-w+u-t-2p_t)\\
&+e(-w+v+u-2p_u+2p_v)+e(-w+v-u-2p_u+2p_v)\\
&+e(v-2p_v+2p_w)+e(-v-2p_v+2p_w),\\
\be_2 &=e(w-2p_w)+e(-w-2p_w),\\
K_1 &=e(-2t-u+w),\\
K_2 &=e(2t-2v-2w).
}

From now on we will ignore $M_K$ (cf. \eqref{MK}) and only work on the $L^2(\R^N)$ part.

\subsection{Twisting}\label{sec:twist}
We have seen from the last section that the action is positive self-adjoint, and each term is very closely related to the standard form $[u]e(-2p_u)$ except for a small twist (cf. first and second term of \eqref{e1}). In the simply-laced case this comes from the commutation relation arising from any relations that are equivalent (i.e. Lusztig transformed) to
$\be_i\cdot \be_j^{\bi b\inv w}$ and locally the action looks exactly like \eqref{e1}. Now using Lemma \ref{qsum} successively, we transform the action by $\be_i \mapsto \Phi_1 \be_i \Phi_1^*,$
$$\Phi_1 = g_b(e(w-u-2p_u+2p_v-2p_w)g_b^*(e(-w-u-2p_u+2p_v-2p_w).$$
This absorb the second term of $\be_1$ to the third term, and spit out the symmetric term from the fourth term, giving
\Eq{\label{a2sym}\be_1 &= [-w+v+u]e(-2p_u) + [v]e(-2p_v+2p_w),\\
\be_2 &=[w]e(-2p_w),\\
K_1 &=e(-2u-v+w),\\
K_2&=e(u-v-2w),}
which is now symmetric in the quantum variables.
On the other hand, one check that this action preserves the action of $\be_2$: it spits out an auxiliary term from the first term, and then reabsorb back to the second term.

For doubly-laced case, this transformation is given by
\Eqn{\Phi_2 &= g_{b_s}(e(-t+v-2p_t+2p_u-2p_v)g_{b_s}^*(e(-t-v-2p_t+2p_u-2p_v)\\
&\circ g_b(e(w-u-2p_u+4p_v-2p_w))g_b^*(e(-w-u-2p_u+4p_v-2p_w)),}
which gives us
\Eq{
\be_1 &= [t+u-w]e(-2p_t)+[u+v-w]e(-2p_u+2p_v)+[v]e(-2p_v+2p_w),\\
\be_2 &=[w]e(-2p_w),\\
K_1 &=e(-2t-u+w),\\
K_2 &=e(2t-2v-2w).\label{b2sym}
}

In general, we can define transformations $\Phi_k$ which preserves the action of all $\be_i, i\neq k$, and bringing the action of $\be_k$ to the standard form \eqref{standardform}. This transformation depends on the fact that there is a variable $v_i$ of minimal index, corresponding to $\be_{ij}$ in \eqref{eijcom} which does not pick up any $e^{\pi bx}$ factor, hence in the standard form \eqref{standardform}. Then one successively apply the above transformation to remove the twisting within $\be_k$ in the order where the commutation relation takes place between adjacent pair of root vectors (or their Lusztig transforms).

Finally, we note that under the following change of variables $T$, with the corresponding action on $p$ given by the transpose inverse:
\Eqn{\veca{u\\v\\w}\mapsto \veca{v-u\\u\\w},\tab \veca{p_u\\p_v\\p_w}\mapsto \veca{p_u+p_v\\p_u\\p_w}}
the representation is exactly the same as that from Proposition \ref{typeA2}. In particular
\begin{Prop} The GNS representation is independent of the choice of the reduced expression of the longest element $w_0\in W$.
\end{Prop}
\begin{proof} For a change of words of $w_0$, it means the change in Lusztig's transformation $...T_iT_jT_i...\to ...T_jT_iT_j...$ which locally is equivalent to the transformation given by the case of type $A_2$. Hence any such transformation is given by the unitary equivalence $T^* \circ \Phi \circ T,$ 
where $\Phi$ is the unitary equivalence from the positive representation \cite{Ip2}, while $T$ is the transformation matrix above. One can then calculate explicitly the map as
\Eq{\Psi = T'\circ g_b(e(w-u+2p_u-2p_v+2p_w)g_b^*(e(-w+u+2p_u-2p_v+2p_w),\label{GNStrans}}
where 
$$T':\veca{u\\v\\w}\mapsto \veca{-u+v+w\\u\\v}$$
is a unitary transformation. Finally note that the transformation $\Psi$ is essentially the inverse transpose of $\Phi$ (interchanging the role of $p_i$'s and $v_i$'s), hence it is consistent with the choice of paths as established in \cite{Ip2}. Similarly we have a generalized transformation for the doubly-laced case given by ratios of 6 quantum dilogarithms followed by a linear transformation, which is again the inverse transpose of the transformation found in \cite{Ip3}.
\end{proof}

With the procedure above, we can construct the GNS representation $\cP_{GNS}$ for $\cU_{q\til{q}}(\fb_\R)$ of arbitrary type (using Proposition \ref{nonsimple_rel2} etc.). In particular, $\cP_{GNS}$ satisfies all properties of a positive representation (positivity and transcendental relations \eqref{transdef}), hence this gives us another family of positive representations for $\cU_{q\til{q}}(\fb_\R)$. From the expression of $\Psi$, which is essentially the inverse transpose of $\Phi$, we have

\begin{Thm}\label{flip} The expressions of $\be_i$ obtained from the GNS construction (after twisting) is precisely the positive representations with all the terms $[X]e(-2p_Y)$ replaced by $[Y]e(-2p_X)$ (cf. Definition \ref{usul}). The action of $K_i$ is simply obtained from a linear change of variables by the Lusztig's isomorphism corresponding to reflections by simple roots.
\end{Thm}

In the rank 2 case, one should compare the expressions \eqref{a2sym}-\eqref{b2sym} with the expressions from Proposition \ref{typeA2} and \ref{B2rep}.
\section{Equivalence of representation $\cP_{GNS}\simeq \cP$}\label{sec:equirep}

So far we have constructed a family of representations $\cP_{GNS}$ from the GNS-representation. In this section we will show that the representation is unitary equivalent to the one constructed earlier in \cite{Ip2, Ip3}.

For type $A_n$, it turns out we can do this directly by a change of variables.
\begin{Thm}\label{mainthm} For type $A_n$, the GNS representation $\cP_{GNS}$ constructed by left multiplication on $\cA=\cU_{q\til{q}}^{C^*}(\fb_\R)$ is unitary equivalent to the positive representation $\cP_\l\simeq \cP$ of $\cU_{q\til{q}}(\g_\R)$ restricted to its Borel part.
\end{Thm}
\begin{proof}
It suffices to consider a particular choice of reduced expression of $w_0$. We fixed the longest element $w_0\in W$ to be the standard expression given in \cite{Ip2}: 
$$w_0 = s_n\;\;s_{n-1}...s_1\;\;s_{n}s_{n-1}...s_2\;\;...\;\;s_{n}s_{n-1}s_n.$$

Then the GNS representation constructed in this section is unitary equivalent to the positive representation $\cP$ by a change of variables (using the notation from Definition \ref{variables})
\Eqn{u_k^1&\mapsto u_{n}^{n+1-k},\\
u_k^m&\mapsto u_{n+1-m}^{n+1-k}-u_{n+2-m}^{n+2-k},\tab m>1.}
\end{proof}

For other types, we use the following trick together with Lusztig's isomorphisms.
\begin{Lem} There exists a unitary transformation such that the action of $\be_i$ is transformed into $q_i\inv K_i \be_i$.
\end{Lem}
\begin{proof}
If we add $\prod_{k=1}^N K_{\a_k}^{t_k}$ in front of $\vec[F]\in\cA$ in Definition \ref{Ubnew}, then the action of $\be_i$ will pick up the correct factors according to $K_i\inv$. But if we commute these factors through the $\be_{\a_k}$ to the right and absorb into the $K$ terms, then on the representation space $\cP$ this is equivalent to multiplying by an exponential function of the form
$e^{\pi \bi h(t_1,...,t_N)},$ where $h(t_1,...,t_N)$ is a homogeneous quadratic function in the variables $t_k$. In particular this is a unitary transformation.
\end{proof}

\begin{Lem} The Lusztig's isomorphism (cf. Theorem \ref{LusThm})  corresponding to the longest word of the Weyl group $w_0=s_{i_1}...s_{i_N}$ given by $T_{i_1}....T_{i_N}$ transform
$$q_i\inv K_i\be_i \mapsto \bf_{\s(i)},$$
where $\bf_i$ is the lower Borel generators, $\s$ is the identity automorphism for $\g$ of type $B_n,C_n, D_{2n}, E_7, E_8, F_4, G_2$, and the unique index 2 automorphisms on the simple roots for $\g$ of type $A_n, D_{2n+1}$ and $E_6$.
\end{Lem}
\begin{proof} It follows directly from the action given by Theorem \ref{LusThm}, the properties from Proposition \ref{LuCox} and the fact that the longest Weyl element acts on simple roots as $\a_i\mapsto -\a_{\s(i)}$.
\end{proof}
Since we have seen from Theorem \ref{flip} that the action given by $\cP_{GNS}$ is precisely obtained by flipping $[X]e(-2p_Y)\to [Y]e(-2p_X)$, the Lusztig's isomorphism $T_i$ modified according to this flip, will map the action of $\be_i$ on $\cP_{GNS}$ to the corresponding action of $\bf_i$ given by Theorem \ref{FKaction}. More precisely, we obtain
\Eq{\be_i=&\sum_{k=1}^n[-u_i^k] e\left(\sum_{j=1}^{v(i,k)-1} a_{i,r(j)}p_j+p_i^k\right),\label{newF}}
with the corresponding roots and reduced expression of $w_0$ relabeled. Recall that $a_{ij}$ is the Cartan matrix, $r(j)$ is the root label corresponding to the variable $p_j$ (which is the shift in the variable $v_j$), and $v(i,k)$ is the index such that $u_i^k=v_{v(i,k)}$ (i.e. we sum all the terms appearing to the left of $p_i^k$) and we used the abuse of notation from Definition \ref{usul}.

\begin{Thm} There exists a unitary transformation such that \eqref{newF} is unitary equivalent to \eqref{FF}.
\end{Thm}
\begin{proof} Notice that for each term in the expression of $\bf_i$, the variable $p_i^k$ with the maximal index $v(i,k)$ is unique, and these cover all variables $(p_1,..., p_N)$.
Therefore we apply the transformation successively from the largest $v(i,k)$ index by
\Eq{b_ip_i^k\mapsto b_ip_i^k -\sum_{j=1}^{v(i,k)-1} a_{i,r(j)}b_jp_j ,}
which induces the corresponding change of variables
\Eq{b_jv_j\mapsto b_jv_j+a_{i,r(j)}b_iv_i^k\label{changevj}}
for $1\leq j \leq v(i,k)-1$. It is clear that under this transformation, we reduce all exponential terms to $e(2p_i^k)$. On the other hand, it is easy to see that under \eqref{changevj} the quantum factors are changed to
$$[-u_i^k]\mapsto \left[- u_i^k-\sum_{j=v(i,k)+1}^N a_{i,r(j)}v_j \right].$$
But this is precisely the action of $\bf_i$ corresponding to the longest element $w_0$ with the reversed reduced expression
$w_0=s_{i_N}...s_{i_1}.$
Since the positive representation is independent of the reduced expression of $w_0$, we proved the unitary equivalence.
\end{proof}
Finally, since $K_i$ depends only linearly on the variables (inside an exponential), it is uniquely determined by the shifts $e(2p_i^k)$ appearing in the action of $\bf_i$. In particular, under the Lusztig's transformation and the above unitary equivalence, $K_i$ from $\cP_{GNS}$ transform precisely to $K_{\s(i)}\inv$ as given in Theorem \ref{FKaction}.

Therefore we can now state our Main Theorem.
\begin{Thm}\label{mainthm} For any simple Lie algebra $\g$, the GNS representation $\cP_{GNS}$ constructed by left multiplication on $\cA=\cU_{q\til{q}}^{C^*}(\fb_\R)$ is unitary equivalent to the positive representation $\cP_\l\simeq \cP$ of $\cU_{q\til{q}}(\g_\R)$ restricted to its Borel part.
\end{Thm}

\begin{Cor}\label{maincor} The tensor product of the positive representations $\cP$ restricted to the Borel part is unitary equivalent to the tensor product of itself with a trivial multiplicity module:
\Eq{W: \cP\ox \cP\simeq M\ox\cP,}
where the underlying Hilbert spaces are all $\cP\simeq M \simeq L^2(\R^N)$, $N=l(w_0)$.

\end{Cor}
\begin{proof} Since we know that $\cP\simeq \cP_{GNS}$ and the GNS-space $\cH\simeq \cP_{GNS}\ox M_K$, this follows immediately from Proposition \ref{WD} of the property of the multiplicative unitary. The unitary transformation will be given by the multiplicative unitary $W$ after factoring the trivial intertwiners on $M_K$, for example, by supplying the delta distributions on the corresponding variables.
\end{proof}

\section{Example} \label{sec:Ex}
In this section, we write explicitly the representations and decompositions for type $A_1$ and $A_2$, which provide us a useful basic framework when we try to deal with the explicit construction of quantum higher Teichm\"{u}ller theory. The computation for type $A_3$ can be found in \cite{Ip6}.
\subsection{Type $A_1$}\label{sec:A1}
Let us first consider the simplest case where $\g$ is of type $A_1$. The Borel part of $\cU_{q\til{q}}(\sl(2,\R))$ in this case is nothing but the quantum plane $\cB_{q\til{q}}$, which is studied extensively in \cite{FK, Ip1}. Nonetheless, it will be illustrative to see how the analysis works with respect to the theory of multiplicative unitary from Proposition \ref{WD}.

First, we can do a transformation to bring the representation $\cP_{GNS}$ into a canonical form. By applying 
$e^{\frac{\pi \bi u^2}{2}}g_b(e(-2u)),$ the action $\be=[u]e(-2p_u)$ is transformed to $\be=e(-2p_u)$. Also take into account the positivity of $K$ under the shift $u\mapsto u+\frac{iQ}{2}$, in terms of the $C^*$-algebraic Borel part $\cA$, the expression is given by
$$\cA \ni f=\int_C f(u,s)\be^{\bi b\inv u}K^{\bi b\inv s}e^{\pi Q s}duds,$$

Now recall from \eqref{WW}, the decomposition is given by the multiplicative unitary
$$W^*(\L(f)\ox \L(g)) = (\L\ox \L)(\D(g)(f\ox 1)).$$
One can compute the action of $W^*$ on $f(u,s)g(v,t)$ to be:
\Eqn{W^*(f\ox g) &= \int f(u-\t, s-t+\t)g(v+\t,t-\t)\frac{G_b(Q+\bi v+\bi \t)e^{-\pi Q t}}{G_b(Q+\bi v)G_b(Q+\bi \t)}e^{-2\pi \bi (u-\t)(t-\t)}d\t\\
&=g_b^*(e(-2v))g_b^*(e(-2p_u-2p_t+2p_v))\circ g_b(e(-2v)) e^{-2\pi \bi t(p_s+u)}(f\ox g),\\
&=g_b^*(e(-2p_u-2p_t+2p_v-2v)\circ g_b^*(e(-2p_u-2p_t+2p_v))\circ e^{-2\pi \bi t(p_s+u)}(f\ox g),}
where we shifted $u,v$ by $\frac{\bi Q}{2}$, according to the Haar measure \eqref{haar} to make the map unitary, and we used Lemma \ref{FT} and the pentagon relation \eqref{qpenta}. Now, recall we can transform it such that the action on $\cP_{GNS}$ does not depend on $s$ and $t$. This can be achieved by taking the Fourier transform on $s$ and $t$, and shift the variables by $u\mapsto u-s, v\mapsto v-t$. We obtain finally
\Eq{W^* =g_b^*(e(-2p_u+2p_v-2v)e^{2\pi \bi p_vu}g_b^*(e(-2p_u-2t))\circ e^{2\pi ip_tu}.}
Since the variables $s,t$ correspond to the trivial module $M_K$, we can obtain the required unitary transformation on $\cP_{GNS}$ by supplying the delta distribution $\d(s)\d(t)$, such that upon acting by $e^{2\pi ip_t u}$, we get
$$W^* = g_b^*(e(-2p_u+2p_v-2v)e^{2\pi \bi p_vu}g_b^*(e(-2p_u+2u)).$$
Finally, acting $W^*$ on $1\ox \cB_q$, the term $g_b^*(e(-2p_u+2u))$ has no effects and acts as a trivial intertwiner. Hence factorizing this action, we simplify $\cP_{GNS}$ and obtain:
\Eq{W^*=g_b^*(e(-2p_u+2p_v-2v))e^{2\pi \bi p_vu}\label{FKgb},}
\Eqn{1\ox \be&\mapsto \D(\be) = \be\ox K+1\ox \be,\\
1\ox K&\mapsto \D(K)=K\ox K,}
which is precisely the integral transformation described in \cite{FK, Ip1}.
\subsection{Type $A_2$}\label{sec:A2}
As we have seen, in general it is hard to calculate the equivalence transformation from the multiplicative unitary. However, using the functional properties of the quantum dilogarithms, the required transformation can be explicitly calculated for Type $A_n$. Let us consider the next simplest case $\g=\sl_3$. First we simplify the expression of $\cP_\l\simeq \cP$ from Proposition \ref{typeA2}.
\begin{Lem} \label{simple}$\cP$ is unitary equivalent to the following expression:
\Eqn{
\be_1=&e(v-w-2p_v)+e(u-2p_u-2p_v+2p_w),&&K_1=e(u-2v+w)),\\
\be_2=&e(w-2p_w)&&K_2=e(-2u+v-2w)).
}
\end{Lem}
\begin{proof}
Apply $\Phi_1$ (i.e. $\be_i \mapsto \Phi_1 \be_i \Phi_1^*$), where
\Eqn{\Phi_1=g_b(e(-u-v+w+2p_u-2p_w))g_b(e(u-v-w-2p_u+2p_w))g_b(e(-2u))g_b(e(-2w)).}
\end{proof}

Using this expression, let us consider the tensor product representation $\cP\ox \cP$:
\Eqn{
\D(\be_1)&=e(v-w+u'-2v'+w'-2p_{v})+e(u+u'-2v'+w'-2p_{u}-2p_{v}+2p_{w})+\\
&\tab e(v'-w'-2p'_v)+e(u'-2p'_u-2p'_v+2p'_w),\\
\D(\be_2)&=e(w-2u'+v'-2w'-2p_{w})+e(w'-2p'_w),\\
\D(K_1)&=e(u-2v+w+u'-2v'+w'),\\
\D(K_2)&=e(-2u+v-2w-2u'+v'-2w'),
}
where the $'$ variables denote the 2nd component.
\begin{Lem} The above action on $\cP\ox \cP$ is unitary equivalent to:
\Eqn{
\D(\be_i)&\simeq 1\ox \be_i, \\
\D(K_i)&\simeq K_i\ox K_i,\tab i=1,2
}
\end{Lem}
\begin{proof}
Apply $\Phi_2$ where
\Eqn{\Phi_2 = &g_b(e(w-u'-2w'-2p_w-2p_u'+4p_w')\circ\\
&g_b(e(u-2v'+w'-2p_u-2p_v+2p_w+2p_u'+2p_v'-2p_w')\circ\\
&g_b(e(v-w-2v'+w'-2p_v+2p_u'+2p_v'-2p_w')\circ\\
&g_b(e(w-2u'+v'-3w'-2p_w+2p_w')).}
\end{proof}

\begin{Cor}\label{repind} From the explicit expression of Lemma \ref{simple}, the transformation $\Phi_2$ can be written in a way independent of the representation:
$$\Phi_2 = g_b(q^2\be_2\ox K_2 \be_{21}\be_{12}\inv \be_2\inv)g_b(q^2\be_{21}\be_2\inv\ox K_1\be_{21}\inv \be_2)g_b(q\be_{12}\be_2\inv\ox K_1 \be_{21}\inv \be_2)g_b(q\be_2\ox K_2\be_2\inv).$$

Hence this transformation gives directly the intertwiners
\Eqn{1\ox \be_i &\simeq \D(\be_i),\\
K_i\ox K_i&\simeq K_i\ox K_i\tab i=1,2.}

Here we put the appropriate $q$ factors to make the arguments positive. Note that the inverse operator is defined in the natural way. Since it is positive and all the terms $q$-commute, the expression of the transformation is still well-defined.
\end{Cor}

Finally, note that $u,v,w$ from first component of $K_i$ now acts as constant, so we can treat them as $\l$'s and use Theorem \ref{removeL} to remove them. We obtain finally
\Eqn{
\D(\be_1)\simeq&1\ox \be_1=e(v'-w'-2p'_v)+e(u'-2p'_u-2p'_v+2p'_w),\\
\D(\be_2)\simeq&1\ox \be_2=e(w'-2p'_w),\\
\D(K_1)\simeq&1\ox K_1=e(u'-2v'+w'),\\
\D(K_2)\simeq&1\ox K_1=e(-2u'+v'-2w').
}
In other words, when restricting to the Borel part, we have
$$\cP_{\l_1}\ox \cP_{\l_2}\simeq \cP\ox \cP=L^2(\R^3)\ox \cP. $$

\begin{Rem}
For type $A_n$, there is a pattern of this series of transformations by $g_b$ computed inductively by Mathematica, and we obtain
$$\cP_{\l_1}\ox \cP_{\l_2}\simeq L^2(\R^{\frac{n(n+1)}{2}})\ox \cP.$$ This procedure seems to be related to the Heisenberg double \cite{Ka1}. For other types however we did not establish any explicit construction of such transformations.
\end{Rem}


\end{document}